\documentclass[a4paper]{amsart}
\usepackage{tikz}
\usepackage{tkz-berge}
\usepackage{tikz-cd}
\usetikzlibrary{quotes, graphs, positioning}
\usetikzlibrary{graphs.standard}

\usepackage{mathtools,amssymb,hyperref}
\usepackage{rotating}

\newcommand{\tmu}{\tilde{\mu}}
\newcommand{\tC}{\widetilde{C}}
\newcommand{\tH}{\widetilde{H}}
\newcommand{\TC}{{T_\CC}}
\newcommand{\TR}{T}

\newcommand{\GC}{\operatorname{Gr}(2,\CC^5)}
\newcommand{\RR}{\mathbb{R}}
\newcommand{\CC}{\mathbb{C}}

\newcommand{\ZZ}{\mathbb{Z}}

\newcommand{\PP}{\mathbb{CP}}
\renewcommand{\P}{\mathbb{P}}
\newcommand{\CO}{\mathcal{O}}
\newcommand{\CA}{\mathcal{A}}

\newcommand{\partialp}{\partial^{^+}\!}
\newcommand{\D}{\mathbb{D}}
\renewcommand{\S}{\mathbb{S}}
\newcommand{\M}{M_\RR}
\newcommand{\rel}{\stackrel{\scriptscriptstyle \phi}{\sim}}
\newcommand{\hrel}{\stackrel{\scriptscriptstyle F_{ij}}{\sim}}
\DeclareMathOperator{\Eff}{\overline{Eff}}

\DeclareMathOperator{\Nef}{Nef}
\DeclareMathOperator{\Proj}{Proj}

\DeclareMathOperator{\Pic}{Pic}

\DeclareMathOperator{\dP}{dP}

\DeclareMathOperator{\id}{id}

\newtheorem{theorem}{Theorem}[section]

\newtheorem{lemma}[theorem]{Lemma}
\newtheorem{proposition}[theorem]{Proposition}

{\theoremstyle{remark}
\newtheorem{remark}[theorem]{Remark}

}

\theoremstyle{definition}

\title[Toric topology of the Grassmannian of planes in $\CC^5$]{Toric topology of the Grassmannian of planes in $\CC^5$ and the del Pezzo surface of degree $5$}
\author[H. S\"u{\ss}]{Hendrik S\"u\ss}
\address{Hendrik S\"u\ss\\ School of Mathematics,
The University of Manchester,
Alan Turing Building,
Oxford Road
Manchester M13 9PL}
\email{\href{mailto:hendrik.suess@manchester.ac.uk}{hendrik.suess@manchester.ac.uk}}
\begin{document}
\begin{abstract}
  We determine the integral homology of the orbit space of a maximal compact torus action on the Grassmannian $\GC$. Our approach uses the well-known Geometric Invariant Theory of the maximal algebraic torus action on this Grassmannian.
\end{abstract}
\maketitle

\section{Introduction}
We consider the Grassmannian $\GC$, parametrising planes in $\CC^5$, together with the Pl\"ucker embedding $\GC \hookrightarrow \P(\bigwedge^2\CC^5)$. The action of the algebraic torus $(\CC^*)^5 \curvearrowright \CC^5$ by diagonal matrices induces an action of the subtorus
\[\TC = \{(t_1,\ldots,t_5) \in (\CC^*)^5 \mid t_1 \cdots t_5 =1 \}\]
 on $\GC$. Via the $\TC$-action on $\CC^5$ we also obtain a linearisation of this action $\TC \curvearrowright \bigwedge^2 \CC^5$. We refer to this linearisation as the \emph{standard} one. In this paper, we are mainly interested in the induced action of the associated compact torus
\[\TR = \{(s_1,\ldots,s_5) \in (\S^1)^5 \mid s_1 \cdots s_5 =1 \}.\]
The orbit space $\GC/\TR$ of this torus action is a compact Hausdorff space and the purpose of this paper is to determine the integral homology groups of this space and more generally to develop a framework which can be applied in similar situations.  The motivation for this problem comes from toric topology where the topology of torus quotients of Grassmannians, flag manifolds and Hessenberg varieties has been studied  recently \cite{zbMATH06783273,2018arXiv180305766B,cgrassmann,Ayzenberg2018,2019arXiv190502294C}, see also \cite{kt-top,suss2018orbit}. Our main result is the following.
\begin{theorem}
\label{thm:main}
  The orbit space $\GC/\TR$ has non-trivial homology groups
  \[
  H_i(\GC/\TR) \cong
  \begin{cases}
    \ZZ & i=0,8\\
    \ZZ/2\ZZ & i=5.
  \end{cases}
  \]
\end{theorem}

\begin{remark}
  The homology groups with coefficients in $\ZZ/2\ZZ$ had been found before by
  Buchstaber and Terzi\'c in the earlier version \cite{cgrassmannp} of their article \cite{cgrassmann}. In between, they also determined the correct integral homology groups independently.
\end{remark}

The proof of Theorem~\ref{thm:main} will take us on a tour from the classical birational geometry of the del Pezzo surface of degree $5$ to the more recent theme of Variation of Geometric Invariant Theory which in turn, by the Kempf-Ness Theorem, links to the topology of the orbits space of the compact torus action. More precisely, our proof of  Theorem~\ref{thm:main} in Section~\ref{sec:homology-groups} will make use of the Geometric Invariant Theory (GIT) of the $\TC$-action on $\GC$. As it will be discussed in Section~\ref{sec:git} this Geometric Invariant Theory is known to be governed by the birational geometry of the degree-$5$ del Pezzo surface. For this reason we first review the birational geometry of this surface in Section~\ref{sec:dp5}. In this context the combinatorial structure of a \emph{directed} chamber decomposition arises naturally. Hence, we start off with introducing this notion in Section~\ref{sec:direct-chamb-decomp} and deriving some of its basic properties.

\subsection*{Acknowledgement}
This research was partially supported by the program 
\emph{Interdisciplinary Research} conducted jointly by the Skolkovo Institute of Science and Technology and the Interdisciplinary Scientific Center J.-V. Poncelet. In particular, I would like to thank Center Poncelet for its great hospitality.
This work was also partially supported by the grant 346300 for IMPAN from the Simons Foundation and the matching 2015-2019 Polish MNiSW fund.

I would like to thank Anton Ayzenberg, Victor Buchstaber, Alexander Perepechko, Nige Ray and Svjetlana Terzi\'c for stimulating discussions on the subject of this paper. Moreover, I am grateful for the helpful comments by an anonymous referee.

\section{Directed chamber decompositions}
\label{sec:direct-chamb-decomp}
Consider a polyhedron $P$ contained in a real vector space $V$. By a \emph{chamber decomposition} of $P$ we understand a polyhedral subdivision induced by a finite set of affine hyperplanes, i.e. the maximal elements (called chambers) of the polyhedral subdivision are given as the closures of connected components of the complement of the hyperplane arrangement.

Now, we equip each of the hyperplanes with a direction, i.e. we consider a finite collection of closed affine halfspaces $\{H^+_\alpha\}_{\alpha \in \CA}$ with $\CA$ being some finite index set and 
\[
  H_\alpha^+ = \{v \in V \mid \langle v, u_\alpha\rangle \geq b_\alpha\} \subset V
\]
for some $u_\alpha \in V^*$ and $b_\alpha \in \RR$. Then $H^-_\alpha = \{v \in V \mid \langle v, u_\alpha\rangle \leq b_\alpha\}$ will denote the opposite closed affine halfspace and $H_\alpha = H^+_\alpha \cap H^-_\alpha$ the separating hyperplane.  Consider a subset $J \subset \CA$ and denote its complement $\CA\setminus J$ by $\overline{J}$. Assume that the corresponding intersection \[P_J = P \cap \left(\bigcap_{\alpha \in J} H_\alpha^+\right) \cap \left(\bigcap_{\alpha \in \overline{J}} H_\alpha^- \right)\]  of positive and negative halfspaces gives a polyhedron of full dimension. These polyhedra are simply the chambers of the chamber decomposition corresponding to the arrangement of hyperplanes $\{H_\alpha\}_{\alpha \in \mathcal{A}}$. Moreover, the boundary of every chamber comes with a distinguished part 
\(\partialp P_J := P_J \cap \left(\partial P \cup \bigcup_{\alpha \in J} H_\alpha\right),\)
which we will call the \emph{positive boundary} of $P_J$. In this situation, we speak of a \emph{directed chamber decomposition} consisting of the \emph{directed chambers} $(P_J,\partialp P_J)$.

If the common intersection of all positive halfspaces
\[P^+:=P_\CA=P \cap \bigcap_{\alpha \in \CA} H_\alpha^+\] has non-empty interior we call it the \emph{central} chamber of the directed chamber decomposition. Note, that by definition we have $\partial P^+ = \partial^+ P^+$.

\begin{lemma}
  \label{lem:pos-boundary-intersect}
  Consider two chambers $P_I, P_J$ with $J \not\subset I$ then $P_I \cap P_J$
  is contained in $\partialp P_J$.
\end{lemma}
\begin{proof}
  By our precondition $J \setminus I$ is not empty. Consider an $\alpha \in J\setminus I$. Then $P_I \cap P_J$ must be contained in both $H^+_\alpha$ and $H^-_\alpha$. Hence, it is a subset of $H_\alpha$. By definition this implies that $P_I \cap P_J$ is a subset of $\partialp P_J$. 
\end{proof}

\begin{lemma}
  \label{lem:pos-boundary-retract}
  Consider a directed chamber decomposition admitting a central chamber. Then
 for every other compact chamber $P_J$ the positive boundary $\partialp P_J$ is a deformation retract of $P_J$.
\end{lemma}
\begin{proof}
Note, that $P$ is itself an intersection of halfspaces. By adding them to the set of positive halfspaces we may assume, that $P=V$, since the chambers of the original decomposition of $P$ are now a subset of the chambers of the induced decomposition of $V$. Moreover, without loss of generality we may assume that $0$ is an interior point of the intersection $\bigcap_{\alpha \in \CA} H_\alpha^+$, since by our precondition an interior point exists and we may apply a global translation moving this point into the origin. In the notation above this means that all the $b_\alpha$ are negative.

We now consider  $v \in P_J$. The intersection of $P_J$ with the line $\RR v$ has to be compact and convex. Hence, we obtain a line segment $L(v)$. Then there is a unique intersection point $r(v):=L(v) \cap \partial^+P_J$. Indeed, we may assume $L(v)=\overline{uw}$, $|u| \geq |w|$. Then $u \in \partial P_J^+$. Indeed, it is clearly a boundary point. Let $H_{\alpha_1},\ldots, H_{\alpha_k}$ be its supporting hyperplanes
i.e. $\langle  u_{\alpha_j} , u \rangle =b_{\alpha_j} < 0$. Then we have
\[\langle  u_{\alpha_j} ,t u \rangle =t\langle u_{\alpha_j} , u \rangle < b_{\alpha_j}\]
for $t>1$. Hence, $t u \in \bigcap_{j=1}^k H^-_{\alpha_j}$. This implies that 
$\alpha_j \in J$ for at least one $j \in \{1,\ldots,k\}$. If not, we would have $t u \in P_J$ for $t$ sufficiently close to $1$. A contradiction to our choice of $u$. 
On the other hand, we have $u \in \bigcap_{\alpha \in J} H^+_\alpha$, i.e. $\langle  u_\alpha , u \rangle \geq b_\alpha < 0$ for all $\alpha \in J$, but this implies \(\langle u_\alpha , s  u \rangle > b_\alpha\) for $s < 1$ and consequently $s u \not\in \bigcup_{\alpha \in J} H_\alpha$. Hence, $L(v) \cap \partialp P_J = u$. Now $H(t,v)= (1-t)\cdot v + t\cdot r(v)$ defines a deformation retraction of $P_J$ to $\partialp P_J$.
\end{proof}

\section{The birational geometry of the del Pezzo surface of degree $5$}
\label{sec:dp5}
In this section we review basic facts on the birational geometry of the del Pezzo surface $Y=\dP_5$ of degree $5$, most of them can be found in \cite[Chapter 8]{zbMATH06053083}.

The del Pezzo surface $Y$ can be constructed as the blowup of $\PP^2$ in $4$ points in general position. Hence, the Picard group $\Pic(Y)$ is a free abelian group of rank $5$ with basis $[E_1], \ldots, [E_5]$, where $E_5$ is the preimage of a line in $\PP^2$ not containing the points of the blowup and $E_1,\ldots,E_4$ are the exceptional divisors of the blowups. The class of the canonical divisor is given by $[K_Y] = -3[E_5] + \sum_{i=1}^4 [E_i]$. The intersection product of curves defines a non-degenerate bilinear form on $\Pic(Y)$, where
\begin{equation*}
  [E_i]\cdot [E_j] =
  \begin{cases}
    1 & i=j=5\\
    -1 & i=j \leq 4\\
    0 & i\neq j.
  \end{cases}
\end{equation*}
The cone of effective divisors $\Eff(Y) \subset N^1 := \Pic(Y) \otimes \RR$ is spanned by the classes of ten irreducible rational curves (i.e. they are isomorphic to $\PP^1$), which we are going to be indexed by two-element subsets $\{i,j\} \subset \{1,\ldots,5\}$. Their classes can be expressed in the basis from above as follows. 
\begin{align*}
  [C_{i5}] & = [E_i], \; i \in \{1,\ldots,4\} \\
  [C_{ij}] & = [E_5]-[E_i]-[E_j], \; \{i,j\} \subset \{1,\ldots,4\}
\end{align*}
These curves are characterised by the properties $[C_{ij}]^2=-1$ and $[K_Y] \cdot [C_{ij}]=-1$. They are usually referred to as $(-1)$-curves or \emph{exceptional} curves. Note, that by the above two such curves $C_{ij}$ and $C_{k\ell}$ intersect if and only if $\{i,j\} \cap \{k,\ell\} = \emptyset$. Hence, the intersection graph of these ten curves is the well-known \emph{Petersen graph} shown in Figure~\ref{fig:petersen}.
\begin{figure}[h]
  \centering
\begin{tikzpicture}[rotate=90]
  \GraphInit[vstyle=classic]
  \SetVertexNoLabel
  \SetUpVertex[MinSize=1pt]
  \grPetersen[RA=1,RB=0.5]
\end{tikzpicture}
  \caption{Intersection graph for the ten exceptional curves}
  \label{fig:petersen}
\end{figure}
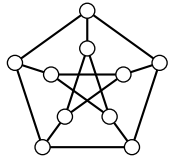

We also consider the polytope $P$ obtained as the cross section
\[P= \Eff(Y) \cap \{v \mid [K_Y] \cdot v = -1\} \subset N^1.\]
Since $\Eff(Y)$ is spanned by the classes  $[C_{ij}]$ which fulfil $[C_{ij}]\cdot [K_Y] = -1$, it follows that $P$ coincides with the convex hull of these classes in $N^1$ and the faces of $\Eff(Y)$ with the exception of the vertex are in one-to-one correspondence with the faces of $P$. Note, that expressing the classes $[C_{ij}]$ in a new basis $b_1,\ldots,b_5$ of $N^1$ given by
$b_i = \frac{1}{2}[E_5] - [E_i]$, we obtain $[C_{ij}]=b_i+b_j$. Hence,
$P$ coincides with the hypersimplex $\Delta_{2,5}$. Furthermore, one observes that the facets of $P$ come in pairs $O_i$ and $T_i$, $i=1,\ldots,5$, where $T_i$ is the tetrahedron obtained as the convex hull of the classes $[C_{ij}]$ with $j \in \{1,\ldots,5\}\setminus \{i\}$ and $O_i$ is the octahedron obtained as the convex hull of the remaining $6$ exceptional classes. Recalling the fact that two exceptional curves intersect if and only if their index sets are disjoint we see that the tetrahedral facets are given exactly by all possible choices of $4$ pairwise disjoint curves among the $10$ exceptional ones.

Now, for every rational divisor class $[D] \in \Eff(Y)$ one obtains a projective variety
\[Y_{[D]} = \Proj \left(\bigoplus_{k \geq 0} H^0(Y,\CO(kD))\right).\]
Here, we have $Y_{[kD]} \cong Y_{[D]}$ for every $k \geq 0$. Hence, we may assume that $[D] \in P$. The construction comes with a contraction map $\phi_{[D]} \colon Y \dashrightarrow Y_{[D]}$. In the case of a surface this rational map extends to a morphism. If we choose $[D]$ from the interior of $P$ the corresponding map $\phi_{[D]}$ is birational inducing an isomorphism between the complement of the exceptional curves and its image.  Moreover, an exceptional curve $C_{ij}$ is contracted to a point under this map if and only if we have $[D] \cdot [C_{ij}] \leq 0$ for the intersection product.  Hence, for every exceptional curve $C_{ij}$ we may define two closed halfspaces and its boundary hyperplane 
 \[
 H_{{ij}}^\pm := \{[D] \mid \pm [D]\cdot [C_{ij}] \geq 0\},\quad H_{{ij}} := \{[D] \mid [D]\cdot [C_{ij}] = 0\}
 \]
 In this way we obtain a directed chamber decomposition of $P$, where the index set $\CA$ for the hyperplanes consists of all exceptional curves (or equivalently of all $2$-element subsets of $\{1,\ldots,5\}$).
 \begin{remark}
   \label{rem:positive-boundary1}
   For $[D] \in P_J\setminus \partialp P_J$ the map
$\phi_{[D]}$ is given as the contraction of the curves in $\overline{J} = \CA\setminus J$.
 \end{remark}
 For some choices of $J \subset \CA$ the intersection $P_J$ has lower dimension, i.e. $P_J$ does not constitute a chamber. The chambers among the $P_J$ correspond exactly to the choice $J \subset \CA$ such that $\overline{J}$ consists of pairwise disjoint exceptional curves. It is not hard to see from the intersection graph that one can choose at most $4$ such curves. The corresponding birational models $Y_{[D]}$ are obtained from $Y$ by blowing down up to $4$ curves. In this way we obtian the blowups of $\PP^2$ in up to $4$ points and additionally $\PP^1 \times \PP^1$.  Note, that there is a central chamber $P^+ \subset P$ spanning the nef cone $\Nef(Y) \subset N^1$. Now, consider a point $[D] \in \partial P$. Then $\phi_{[D]} \colon Y \to \PP^1$ defines a conic bundle structure on $Y$ with three reducible fibres consisting of two intersecting exceptional curves each if $[D]$ lies in the interior of one of the facets $O_k$. Here, these six exceptional curves are exactly the ones corresponding to the vertices of $O_k$. All other exceptional curves are mapped isomorphically onto $\PP^1$.  Finally, if $[D]$ does not lie in the interior of one of the octahedral facets, then $\phi_{[D]} \colon Y \to *$ is the contraction to a point.
 
 \section{The Geometric Invariant Theory of $\GC$}
 \label{sec:git}
 In this section we use standard results from the variation of GIT quotients. We refer to \cite{zbMATH04029746, zbMATH00108945,zbMATH01560390} for details.  

The standard linearisation of the torus action on $\GC$ defines a moment map on $\GC$ as follows.
  \begin{align*}
    \mu \colon \GC &\to \M := \RR^5/\RR\cdot (1,1,1,1,1) \cong \RR^4 ;\\
    (z_{ij}) &\mapsto \frac{\sum_{ij} |z_{ij}|^2 (e_i + e_j)}{\sum_{ij}|z_{ij}|^2} + \RR\cdot (1,1,1,1,1).
  \end{align*}
It is easy to see that this map is invariant under the compact torus action. Hence, we obtain an induced map $\tmu \colon \GC/\TR \to \M$ and a fibre of this map can be naturally identified 
with the quotient $\mu^{-1}(u)/\TR$. On the other hand, the fibre quotients $\mu^{-1}(u)/\TR$ are known to be homeomorphic to quotients
of $\GC$ by the $\TC$-action in the sense of GIT. In particular, these fibres are homeomorphic to algebraic varieties. The fibre $\tmu^{-1}(0)$ equals the GIT quotient corresponding to the standard linearisation.

The GIT quotients $Y_u := \tmu^{-1}(u) = \mu^{-1}(u)/\TR$ are known to vary in a rather well-behaved manner. Namely, there is a polyhedral subdivision $\Lambda$ of the moment polytope $P$, such that 
$\tmu^{-1}(u) \cong \tmu^{-1}(u')$ whenever there is an element $\lambda \in \Lambda$ such that $u$ and $u'$ belong to the relative interior $\lambda^\circ$ of $\lambda$, see e.g. \cite[§3]{zbMATH04029746}.  Hence, we may define $Y_{\lambda} := Y_{u}$ for any $u$ in the relative interior of $\lambda$. The map $\tmu$ induces a trivial fibration over the relative interior of each $\lambda \in \Lambda$, i.e. $\tmu^{-1}(\lambda^\circ) \cong \lambda^\circ \times Y_\lambda$. Moreover, for every pair of faces 
$\lambda' \prec \lambda$ there is an induced contraction morphism $f_{\lambda\lambda'} \colon Y_\lambda \to Y_{\lambda'}$. Hence, we obtain an inverse system of GIT quotients indexed by the elements of the polyhedral subdivision $\Lambda$.

The GIT quotient for the standard linearisation of the $\TC$-action on $\GC$ is known to be $Y=\dP_5$, the del Pezzo surfaces of degree $5$. This is shown in \cite[Prop.~3.2]{zbMATH00553990}. The proof there also shows that the complement of the semi-stable locus has codimension $2$. Moreover, the stable and semi-stable locus coincide by Corollary 2.5 loc.cit. Now, by Theorem 2.3, Remark 2.3.1 and Corollary~2.4 of \cite{zbMATH01700876} we conclude, that birational geometry of the del Pezzo surface $Y$ and the Geometric Invariant Theory of the Grassmannian $\GC$ coincide in the following sense.

There is an isomorphism of vector spaces
  \(\Psi \colon  \M \times \RR \to N^1(Y)\)
  identifying the cone over $P\times \{1\}$ with $\Eff(Y)$, such that
\begin{enumerate}
\item  the polyhedral subdivision $\Lambda$ gets identified with the chamber decomposition from Section~\ref{sec:dp5}.
\item
  for a fixed $\lambda \in \Lambda$ and any two rational elements $u,u' \in \lambda^\circ$ in its relative interior the corresponding contractions $\phi_{\Psi(u,1)} \colon Y \to Y_{\Psi(u,1)}$ and $\phi_{\Psi(u',1)} \colon Y \to Y_{\Psi(u',1)}$ coincide,
\item 
  for $\lambda \in \Lambda$ and $u \in \lambda^\circ$ the GIT quotient $Y_\lambda$ coincides with the birational model $Y_{\Psi(u,1)}$. Hence, we obtain well-defined morphisms  $\phi_\lambda \colon Y \to Y_\lambda$ by setting $\phi_\lambda = \phi_{\Psi(u,1)}$. 
\item Moreover, the contraction morphisms $\phi_\lambda$ are compatible with the inverse system of GIT quotients, i.e. for $\lambda' \prec \lambda$ we have $\phi_{\lambda'}=f_{\lambda\lambda'} \circ \phi_\lambda.$ \label{item:inv-system-compatible}
\end{enumerate}

Goresky and MacPherson showed in \cite[§5]{zbMATH04029746} that the orbit space
$X/\TR$ of a projective variety $X \subset \PP^N$ with moment polytope $P$ can be recovered from the corresponding inverse system of GIT quotient as the identification space
 \[X/\TR \cong \left(\bigsqcup_{\lambda \in \Lambda} \lambda \times Y_\lambda\right)/{\sim},\]
 where $(u,y) \sim (u,y')$ if $(u,y') \in \lambda' \times Y_{\lambda'}$,
 $(u,y) \in \lambda \times Y_\lambda$ with $\lambda' \prec \lambda$ 
 and $f_{\lambda\lambda'}(y)=y'$.

 In our setting, where $Y$ dominates all GIT quotients of $\GC$ via the morphisms $\phi_\lambda \colon Y \to Y_\lambda$, this construction can be simplified as follows. Let us define an equivalence relation $\rel$ on $P \times Y$ by
\[(u,y) \rel (u,y')  \Leftrightarrow \phi_\lambda(y) = \phi_\lambda(y')\]
for the unique $\lambda \in \Lambda$ with $u \in \lambda^\circ$. Then by \cite[Corollary 2.2]{suss2018orbit} we have
\begin{equation}
  \label{eq:ident-space}
  \GC/\TR \cong (P \times Y)/{\rel}.
\end{equation}
For the remaining part of this paper we will utilise this homeomorphism to study the topology of $\GC/\TR$.

From now one we will identify $P$ with its image $\Psi(P\times \{1\})$ and the polyhedral subdivision $\Lambda$ with the one induced by the directed chamber decomposition of Section~\ref{sec:dp5}. In the following Proposition we merely reformulate the observations from Section~\ref{sec:dp5} in terms of GIT and the map $\tmu \colon \GC/\TR \to P$.

\begin{proposition}
  \label{prop:GIT-chambers}
  For an element $\lambda \in \Lambda$ of maximal dimension we have a homeomorphism
  \(\tmu^{-1}(\lambda \setminus \partialp \lambda) \cong (\lambda \setminus \partialp \lambda) \times Y_\lambda,\)
  where $Y_\lambda$ is either $\PP^1 \times \PP^1$ or a blowup of $\PP^2$ in up to $4$ points.

  Morover, for the facets $O_i \prec P$ we have $\tmu^{-1}(O_i^\circ) \cong O_i^\circ \times \PP^1$. For the second type of facets $T_i \prec P$ we have $\tmu^{-1}(T_i) \cong T_i$.
\end{proposition}

\section{The homology groups for $\GC/\TR$}
\label{sec:homology-groups}
We consider the chamber decomposition $\Lambda$ of the hypersimplex $P=\Delta_{2,5}$ from Section~\ref{sec:dp5}, which parametrises the contractions $Y_{[D]}$ of $Y=\dP_5$. In Section~\ref{sec:git} we have seen that that these contractions coincide with the fibres $Y_u = \tmu^{-1}(u)$.
For $i=1,\ldots,5$ let $\Lambda_k$ the set of chambers $\lambda \in \Lambda$ such
that for $u \in \lambda^\circ$ the morphism $Y \to Y_u$ contracts exactly $5-k$ of the $10$ exceptional curves. In the notation of Section~\ref{sec:dp5} this
is 
\(\Lambda_k = \{P_J \mid \# \overline{J} = 5-k\}.\)
Let $|\Lambda_k|$ denote the support of $\Lambda_k$, i.e. $|\Lambda_k| := \bigcup_{\lambda \in \Lambda_k} \lambda$. By Lemma~\ref{lem:pos-boundary-intersect} for a chamber $\lambda \in \Lambda_{k+1}$ we have $\lambda \cap \bigcup_{i=1}^k|\Lambda_{i}| = \partialp\lambda$. Now, for $k=1, \ldots, 5$ we set 
\[V_k=\tmu^{-1}\left(\bigcup_{i=0}^k |\Lambda_i|\right)\]
and $V_0 = \tmu^{-1}(\partial P)$. In this way we obtain a filtration of $\GC/\TR$ by closed subspaces. Our strategy to determine the homology groups of $V_5=\GC/\TR$ is as follows.
\begin{enumerate}
\item In Proposition~\ref{prop:relative-hom-V5/V4-V5/V0} we show that the relative homology groups of the pairs $(V_5,V_{4})$ and $(V_5,V_{0})$ coincide.
\item In Lemma~\ref{lem:relative-hom-V5/V4} we calculate the relative homology groups of the pair $(V_5,V_4)$.
\item In Lemma~\ref{lem:V0} we calculate the homology groups of $V_0$.
\item We consider the long exact sequence of homology for the pair $(V_5,V_0)$. The results from steps (1)-(3) allow us to calculate most of the homology groups of $V_5$.
\item To determine the remaining homology groups we have to study some of the homomorphisms in the long exact sequence in detail. This is done in Lemma~\ref{lem:inclusion} and Proposition~\ref{prop:boundary-map}.
\end{enumerate}

\goodbreak
\begin{lemma}
  \label{lem:contractible}
  For every full dimensional $\lambda \in \Lambda$ the quotient space $\tilde\mu^{-1}(\lambda)/\tilde\mu^{-1}(\partialp \lambda)$ is contractible. 
\end{lemma}
\begin{proof}
  First note, that  \(\tilde\mu^{-1}(\lambda)/\tilde\mu^{-1}(\partialp \lambda)
    \;\cong\; \left(\lambda \times Y_\lambda\right)/ \left(\partialp \lambda \times Y_\lambda\right)\)
    by Proposition~\ref{prop:GIT-chambers}. Now Lemma~\ref{lem:pos-boundary-retract} says that  $\partialp \lambda \subset \lambda$ is a deformation retract of $\lambda$. Let $H \colon [0,1] \times \lambda \to \lambda$ be the deformation retraction.  Then
    \[H\!\times\!{\id} \colon ([0,1] \times \lambda) \times Y_\lambda \to \lambda \times Y_\lambda\] is a deformation retraction
    of $\lambda \times Y_\lambda$ onto  $\partialp \lambda \times Y_\lambda$
    and finally the composition of $H\!\times\!\id$ with the quotient map $\left(\lambda \times Y_\lambda\right) \to \left(\lambda \times Y_\lambda\right)/ \left(\partialp \lambda \times Y_\lambda\right)$ gives a deformation retraction of $\left(\lambda \times Y_\lambda\right)/ \left(\partialp \lambda \times Y_\lambda\right)$ onto the point $\left(\partialp \lambda \times Y_\lambda\right)/\left(\partialp \lambda \times Y_\lambda\right)$.
\end{proof}

\begin{proposition}
  \label{prop:relative-hom-V5/V4-V5/V0}
  The identity on $V_5$ induces an isomorphism of relative homology groups
  \(  
    H_i(V_{5},V_{4}) \cong H_i(V_{5},V_{0}).
  \)
\end{proposition}
\begin{proof}
  We show that $H_i(V_{5},V_{4}) \cong H_i(V_{5},V_{5-k})$ by induction on $k=1,\ldots,5$.
  Assume we know that $H_i(V_{5},V_{4}) \cong H_i(V_{5},V_{5-k})$ then we consider the pair
  $V_{5-k-1} \subset V_{5-k}$ for its relative homology we have
  $H_i(V_{5-k-1},V_{5-k}) \cong \tH_i(V_{5-k-1}/V_{5-k})$, but then
   \[V_{k+1}/V_{k} \cong \bigvee_{\lambda \in \Lambda_{k+1}} \tmu^{-1}(\lambda)/\tmu^{-1}(\partialp \lambda), \]
   is contractible by Lemma~\ref{lem:contractible}. Hence, we obtain 
   $H_i(V_{5-k-1},V_{5-k})=0$. Now, the long exact sequence for the triple 
   $V_{5-k-1} \subset V_{5-k} \subset V_5$ implies that
   $H_i(V_5,V_{5-k-1}) \cong  H_i(V_{5},V_{5-k})$ and we obtain $H_i(V_5,V_{5-k-1}) \cong H_i(V_{5},V_{4})$ by induction hypothesis.
\end{proof}

Note, that by definition $V_4$ coincides with  the complement of $\tilde{\mu}^{-1}((P^+)^\circ)$ in $V_{5}=\GC/\TR$. Hence, we have $V_{5}/V_4\cong (P^+ \times Y)/(\partial P^+ \times Y)$.

\begin{lemma}
  \label{lem:relative-hom-V5/V4}
  The homology groups of  $V_{5}/V_4$ are given by
  \[
  H_i(V_5/V_4) \cong
  \begin{cases}
    \ZZ & i=4,8,0\\
    \ZZ^5 & i=6,\\
    0 & \text{else}
  \end{cases}
\]
Moreover,  $H_6(V_5/V_4)$ is generated by the images
of the homomorphisms
\[(\id \times\,\iota^{ij})_* \colon H_6((P^+ \times C_{ij})/(\partial P^+ \times C_{ij})) \to H_6(V_5/V_4),\]
where $1 \leq i < j \leq 5$ and $\iota^{ij} \colon C_{ij} \hookrightarrow Y$ denotes the inclusion of the $(-1)$-curve $C_{ij}$.
\end{lemma}
\begin{proof}
The non-trivial homology groups for $Y$, the blowup of $\PP^2$ in four points, are known to be 
\[
H_i(Y) \cong
\begin{cases}
  \ZZ & i=0,4\\
  \ZZ^5 & i=2.
\end{cases}
\]
Moreover, the homology group $H_2(Y) \cong \Pic(Y)$ is generated by
the fundamental classes of the $(-1)$-curves $C_{ij}$.

Applying the K\"unneth theorem to products of pairs $(P^+,\partial P^+) \times (C_{ij},\emptyset)$ and $(P^+,\partial P^+) \times (Y,\emptyset)$, respectively we obtain the following commutative diagram
\begin{equation}
\begin{tikzcd}
\bigoplus_{r+s=m} H_r(P^+, \partial P^+) \otimes H_s(C_{ij}, \emptyset) \ar[d,"\id_* \otimes \, \iota^{ij}_*"]\ar[r,"\sim"] &   \ar[d,"(\id \times\, \iota^{ij})_*"] H_m(P^+ \times C_{ij}, \partial P^+\times C_{ij})\\
\bigoplus_{r+s=m} H_r(P^+, \partial P^+) \otimes H_s(Y, \emptyset) \ar[r,"\sim"] &  H_m(P^+ \times Y,\partial P^+ \times Y)
\end{tikzcd}\label{eq:kuenneth}
\end{equation}

Here, the horizontal isomorphisms are just the relative K\"unneth isomorphisms and the commutativity of the diagram just corresponds to the naturality of these isomorphisms. From the lower row of this diagram we directly conclude
\begin{align*}
 \tH_i(V_5/V_4) \cong H_i(P^+ \times Y,\partial P^+ \times Y) 
\cong
 \begin{cases}
   \ZZ & i=4,8\\
   \ZZ^5 & i=6\\
   0  & \text{else}.
 \end{cases}
\end{align*}
Moreover, the only nontrivial summands in
\[H_6(P^+ \times Y,\partial P^+ \times Y) \cong \bigoplus_{r+s=6} H_r(P^+, \partial P^+) \otimes H_s(Y, \emptyset)\]
and
\[H_6(P^+ \times C_{ij},\partial P^+ \times Y) \cong \bigoplus_{r+s=6} H_r(P^+, \partial P^+) \otimes H_s(C_{ij}, \emptyset)\]
are  $H_4(P^+, \partial P^+) \otimes H_2(Y, \emptyset)$ 
and $H_4(P^+, \partial P^+) \otimes H_2(C_{ij}, \emptyset)$, respectively. Since $H_2(Y)$ is generated by the images
of the homomorphisms \[\iota^{ij}_* \colon H_2(C_{ij}) \to  H_2(Y)\] it follows by (\ref{eq:kuenneth}) that
$H_6(P^+ \times Y,\partial P^+ \times Y)$ is generated by the images
of the homomorphisms \[(\id \times\, \iota^{ij})_* \colon H_6(P^+ \times C_{ij},\partial P^+ \times C_{ij}) \to H_6(P^+ \times Y,\partial P^+ \times Y).\]
Composing again with the map of pairs \[(P^+ \times Y,\partial P^+ \times Y) \to
  ((P^+ \times Y)/(\partial P^+ \times Y),*) \cong (V_5/V_4,*)\]
implies our claim.
\end{proof}

Lemma~\ref{lem:relative-hom-V5/V4} in combination with Proposition~\ref{prop:relative-hom-V5/V4-V5/V0} provides us with the relative homology groups of the pair $(V_5,V_0)$. Our next aims are to calculate the homology groups of $V_0$ and then use the long exact sequence of the pair $(V_5,V_0)$ to eventually determine the homology groups of $V_5=\GC/\TR$.

To study the homology groups of $V_0$ we follow the approach of \cite[12.1]{cgrassmann} and consider a section $s \colon P \to \GC/\TR$ of $\tmu$. Such a section can be constructed as follows. Fix an element $y \in Y$ and set
$s(u)=[(u,y)] \in (P \times Y)/{\rel}$.
Recall from Proposition~\ref{prop:GIT-chambers} that $\tmu^{-1}(O_i^\circ) \cong O_i^\circ \times \S^2$ and
$\tmu^{-1}(u) = \{s(u)\}$ for $u \not\in \bigcup_i  O_i^\circ$. 
Then $\tmu^{-1}(O_i^\circ) \setminus s(\partial P) \cong  O_i^\circ \times (\S^2 \setminus \{*\})$ as a product of open balls is itself homeomorphic to an open ball and its one-point compactification is a $5$-sphere, which we denote by $S^5_i$. On the other hand, we have
\[V_0 \cong s(\partial P) \cup \coprod_{i=1}^5 \left(\tmu^{-1}(O_i^\circ) \setminus s(\partial P)\right).\]
Note, that $V_0$ as a closed subset of a compact space is itself compact. Hence,
$V_0/s(\partial P)$ is compact as well. It follows that $V_0/s(\partial P)$ is the one-point compactification of $\coprod_{i=1}^5 \left(\tmu^{-1}(O_i^\circ) \setminus s(\partial P)\right)$, which by the above is homeomorphic $\bigvee_{i=1}^5 S^5_i$.


\begin{lemma}[{\cite[Lem.~29]{cgrassmann}}]
  \label{lem:V0}
  The non-trivial homology groups of $V_0$ are the following
\[
H_i(V_0,\ZZ) =
\begin{cases}
  \ZZ & i=0,3\\
  \ZZ^5 & i=5.
\end{cases}
\]
Moreover,
the homomorphism $H_3(\partial P) \to H_3(V_0)$ and
$H_5(V_0) \to H_5\left(\bigvee_i S_i^5\right)$, induced by the section $s$ and
the quotient map $V_0 \to V_0/s(\partial P)$, respectively, are isomorphisms.
\end{lemma}
\begin{proof}
  This follows directly from the considerations above, the fact that $\partial P$ is homeomorphic to $\S^3$ and the long exact sequence for the pair $(V_0, s(\partial P))$.
\end{proof}


\begin{proof}[Proof of Theorem~\ref{thm:main}]
The exact homology sequence for the pair $V_0 \subset V_5$ together with Proposition \ref{prop:relative-hom-V5/V4-V5/V0} and Lemma~\ref{lem:relative-hom-V5/V4} implies $H_i(V_5)=0$ for $i=1,2,7,8$. Moreover, we obtain the exact sequences
\begin{equation}
0 \to H_4(V_5) \longrightarrow \ZZ \longrightarrow \ZZ \stackrel{\iota_*}{\longrightarrow} H_3(V_5) \to 0.\label{eq:exact-seq-h4}
\end{equation}
and
\begin{equation}
0 \to H_6(V_5) \longrightarrow \ZZ^5 \stackrel{\partial_6}{\longrightarrow} \ZZ^5 \longrightarrow H_{5}(V_5) \to 0 \label{eq:exact-seq-h6}
\end{equation}
Now by Lemma~\ref{lem:inclusion}, below, the inclusion $\iota \colon V_0 \hookrightarrow V_5$ induces the trivial homomorphism $H_3(V_0) \to H_3(V_5)$.
It follows from the exactness of (\ref{eq:exact-seq-h4}) that $H_3(V_5)=0$ and that the central homomorphisms in (\ref{eq:exact-seq-h4}) is an isomorphism and eventually that $H_4(V_5) =0$.

We will show later in Proposition~\ref{prop:boundary-map} that the co-kernel of
the connecting homomorphism $\partial_6 \colon H_6(V_5,V_0) \to H_5(V_0)$ is
$\ZZ/2\ZZ$. Then the exactness of (\ref{eq:exact-seq-h6}) implies $H_5(V_5)\cong \ZZ/2\ZZ$ and $H_6(V_5)=0$.
\end{proof}

\begin{lemma}
  \label{lem:inclusion}
  Let $\iota \colon V_0 \to V_5 = \GC/\TR$ denote the inclusion map. Then
  the induced map $\iota_*\colon H_3(V_0) \to H_3(V_5)$ is trivial.
\end{lemma}
\begin{proof}  
  Consider the section $s \colon P \to V_5$ of $\tilde\mu$ from above. Then by Lemma~\ref{lem:V0} the induced homomorphism $(s|_{\partial P})_* \colon H_3(\partial P) \to H_3(V_0)$ is an isomorphism. On the other hand, we have the following commutative diagram.
\[\begin{tikzcd}
V_0 \arrow[r, hook, "\iota"] & V_5  \\
\partial P  \arrow[u,"s"] \arrow[r, hook] & P \arrow[u,"s"]
\end{tikzcd}\]
It follows that $\iota_* \colon H_3(V_0) \to H_3(V_5)$ must be trivial, as $H_3(P)=0$.
\end{proof}

In the following for a product $Q \times Z$ and $F \subset Q$ we define an equivalence relation as follows. For $u \in F$ we have
$(u,z) \stackrel{\scriptscriptstyle F}{\sim} (u,z')$ for all $z \in Z$. With this notation we have the following technical lemma
\begin{lemma}
  \label{lem:disc-homeo}
  Let $H^+ = \RR^{k-1} \times \RR_{\geq 0}$ be a closed halfspace with boundary hyperplane $H= \RR^{k-1} \times \{0\}$. Consider $\D^+ = \D^k \cap H^+$, $\S^+ = \S^{k-1}$ and $\D^{k-1} = \D^k \cap H$. Then there exists a homeomorphism
  \((\D^+ \times \S^m)/\,{\stackrel{\scriptscriptstyle\S^+}{\sim}} \cong \D^{k+m}\)  identifying $(\D^{k-1} \times \S^m)/{\stackrel{\,\scriptscriptstyle\S^+}{\sim}}$ with the boundary sphere $\S^{k+m-1}$.
\end{lemma}
\begin{proof}
The map
  \[
  (\D^+ \times S^m)/ {\stackrel{\scriptscriptstyle\S^+}{\sim}} \to \D^{k+m}, \quad [(u,y)] \mapsto
  (u_1,\sqrt{1-|u|^2}\cdot y)
  \]
  defines a homeomorphisms identifying $(\D^{k-1} \times \S^m)/{\stackrel{\scriptscriptstyle\S^+}{\sim}}$ with the boundary sphere. Here, for $u \in \RR^{k-1} \times \RR$ we denote by $u_1$ and $u_2$ the projection to the first and second factor, respectively. 
\end{proof}

For every exceptional curve $C_{ij}\cong \PP^1 \cong \S^2$ we consider the polytopes $P_{ij}=H^+_{ij} \cap P$ and its facet $F_{ij}=P \cap H_{ij}$.
Then the curve $C_{ij}$ does \emph{not} get contracted under the map
$Y \to Y_u$ for $u \in P_{ij}^\circ$, but it gets contracted under this map if $u \in H^-_{ij} \cap P$. If $u \in \partial P$ then $C_{ij}$ gets \emph{not} contracted if and only if $u \in O_i^\circ$  or $u \in O_j^\circ$. Now, the inclusion $P_{ij} \times C_{ij} \subset P \times Y$ induces an embedding
\[\widetilde{C}_{ij} := {(P_{ij} \times C_{ij})/{\rel}} \lhook\joinrel\longrightarrow \; {(P \times Y)/{\rel}} \; \cong\;  \GC/\TR.\]
The pair $(P_{ij},F_{ij})$  is homeomorphic to the pair $(\D^+,\S^+)$ in Lemma~\ref{lem:disc-homeo}. Hence, by that Lemma we conclude that $D^6_{ij} := (P_{ij} \times C_{ij})/{\hrel}$ is homeomorphic to the closed disc $\D^6$, where the boundary $\S^5 \subset\D^6$ is being identified with  the sphere $S^5_{ij}:=((P_{ij} \cap \partial P) \times C_{ij})/{\hrel}$. Recall that $C_{ij}$ gets contracted under the map $Y \to Y_u$ for $u \in H_{ij}$, i.e. for $u \in H_{ij}$ we have $(u,y) \rel (u,y')$ for every two points $y,y' \in C_{ij}$. Hence, we obtain a map $\Psi^{ij} \colon D^6_{ij} \to V_5$ via the composition
\[\underbrace{(P_{ij}\times C_{ij})/{\hrel}}_{= D^6_{ij}} \;\rightarrow\; \underbrace{(P_{ij}\times C_{ij})/{\rel}}_{= \tC_{ij}} \;\hookrightarrow\; V_5\]
and $S^5_{ij} \subset D^6_{ij}$ is sent to $\tC_{ij} \cap V_0$ via this map. Hence the $\Psi^{ij}$ are maps of pairs $(D^6_{ij},S^5_{ij}) \to (V_5,V_0)$ for $1 \leq i < j \leq 5$.
These maps induce homomorphisms between the corresponding long exact sequences fitting into the following commutative diagram.
\begin{equation}
\begin{tikzcd}
\arrow[r] & H_6(S^5_{ij})  \arrow[r]\arrow[d,"\Psi^{ij}_*"] &  H_6(D^6_{ij}) \arrow[r]\arrow[d,"\Psi^{ij}_*"]  &  H_6(D^6_{ij},S^5_{ij})  \arrow[r,"\cong"]\arrow[d,"\Psi^{ij}_*"]  & H_{5}(S^5_{ij}) \arrow[r]\arrow[d,"\Psi^{ij}_*"] & \  \\
\arrow[r] & H_6(V_0)  \arrow[r] &  H_6(V_5) \arrow[r]  &  H_6(V_5,V_0)  \arrow[r]  & H_{5}(V_0) \arrow[r] & \  \\
\end{tikzcd}\label{eq:map-of-sequences}
\end{equation}

\begin{lemma}
  \label{lem:induced-orientation}
  Consider a continuous map of pairs $q \colon (X,\partial X) \to (Y/\partial Y,*)$. Where $X$ and $Y$ are compact orientable $n$-dimensional manifold with boundary. Assume that there exist an open subset $U \subset X \setminus \partial X$ which is mapped homeomorphically onto some open subset $V \subset Y/\partial Y$.
  
  Then $q$ induces an isomorphism
  \[q_* \colon H_n(X,\partial X) \xrightarrow{\sim}  \tH_n(Y/\partial Y).\]
\end{lemma}
\begin{proof}

We set $Z=Y/\partial Y$ and pick a point $y \in Y^\circ:=Y \setminus \partial Y$, which we identify with its image in $Z$. We than have $H_n(Z, Z \setminus y)\cong \tH_n(Z)$. Indeed consider the following diagram
\[
     \begin{tikzcd}
       H_n(Y,\partial Y) \arrow[d,"\rotatebox{90}{\(\sim\)}"'] \arrow[r,"\sim"] & \ar[d] H_n(Y,Y\setminus y) &\arrow[l,"\sim"'] \ar["\rotatebox{90}{\(\sim\)}"']{d} H_n(Y^\circ, Y^\circ \setminus y)\\
       \tH_n(Z) \arrow[r] &  H_n(Z,Z\setminus y) & \arrow[l,"\sim"'] H_n(Y^\circ, Y^\circ \setminus y)
     \end{tikzcd}       
\]
Both sides of the diagram commute as the homomorphisms are induced by inclusions of pairs and the quotient map. The horizontal arrows on the right are isomorphisms by excision. By commutativity of the right square the central vertical arrow has to be an isomorphism and the commutativity of the left square implies that the lower left horizontal arrow is an isomorphism as well.  

We argue similarly as before. For every $x \in U \cap \phi^{-1}(Y^\circ)$ we get essentially the same diagram as above
   \[
     \begin{tikzcd}
       H_n(X,\partial X) \arrow[d,"q_*"] \arrow[r,"\sim"] & \ar[d,"q_*"] H_n(X,X\setminus x) &\arrow[l,"\sim"'] \ar["\rotatebox{90}{\(\sim\)}"',"q_*"]{d} H_n(U, U\setminus x)\\
       \tH_n(Z) \arrow[r,"\sim"] &  H_n(Z,Z\setminus q(x)) & \arrow[l,"\sim"'] H_n(V, V\setminus q(x))
     \end{tikzcd}       
   \]
with both squares commuting. This time we know that the horizontal arrow in the lower left is an isomorphism due to the above considerations. Arguing as before shows that the left vertical arrow is an isomorphism.
\end{proof}

\begin{lemma}
  \label{lem:generators}
  The classes $\Psi^{ij}_*[D^6_{ij}, S^5_{ij}]$ generate the relative homology group $H_6(V_5,V_0)$.
\end{lemma}
\begin{proof}
  Recall, that $D^6_{ij}=P_{ij} \times C_{ij}/{\hrel}$. By the definition of $P^+$ we have $P^+ \subset P_{ij}$ and $H_{ij} \cap P^+ \subset \partial P^+$. Hence, $(P^+ \times C_{ij})/(\partial P^+ \times C_{ij})$ can be identified with the quotient of $D^6_{ij}$ by the complement of $(P^+)^\circ \times C_{ij}$ and  we obtain a map of pairs
  \[\phi \colon (D^6_{ij}, S^5_{ij}) \to ((P^+ \times C_{ij})/(\partial P^+ \times C_{ij}),*).\]
  Moreover, this map fits into the following commutative diagram
 \[\begin{tikzcd}
     (D^6_{ij}, S^5_{ij}) \ar[d,"\Psi^{ij}"] \ar[rr,"\phi"]&\ &((P^+ \times C_{ij})/(\partial P^+ \times C_{ij}),*) \ar[d,"\id \times\, \iota^{ij}"] &\\
 (V_5,V_0)  \ar[r,hook] & (V_5,V_4) \ar[r] & (V_5/V_4,*) \cong ((P^+ \times Y)/(\partial P^+ \times Y),*)
\end{tikzcd}
\]
By Proposition~\ref{prop:relative-hom-V5/V4-V5/V0} we see that in the induced diagram on the level of homology we have isomorphisms in the bottom row. In the top row we have at least an isomorphism for $H_6$ by Lemma~\ref{lem:induced-orientation} as $(P^+)^\circ \times C_{ij} \subset D^6_{ij}$ is an open subset which gets mapped homeomorphically onto it image in $(P^+ \times C_{ij})/(\partial P^+ \times C_{ij})$.

By Lemma~\ref{lem:relative-hom-V5/V4} we know that $\tH_6(V_5/V_4)$ is generated by images of the homology groups $\tH_6((P^+ \times C_{ij})/(\partial P^+ \times C_{ij}))$. By the commutativity of the diagram we see that the images
of the groups $H_6(D^6_{ij},S^5_{ij})$ have to generate $H_6(V_5,V_0)$.
\end{proof}

\begin{proposition}
  \label{prop:boundary-map}
  The connecting homomorphism $\partial_6$ of the long exact sequence
 \[\ldots \to H_6(V_0) \to H_6(V_5) \to H_6(V_5,V_0) \stackrel{\partial_6}{\longrightarrow} H_5(V_0) \to \ldots \ \]
  has cokernel $\ZZ/2\ZZ$ .
\end{proposition}
\begin{proof}
  Recall how the map $\partial_6$ is constructed. A relative cycle
  of the pair $(V_5,V_0)$ is just mapped to its boundary, which by definition lies in $V_0$. In our case we consider the cycles $\Psi^{ij}_*[D^6_{ij},S^5_{ij}]$. The image under the connecting homomorphism of such a relative cycle is just given by $\Psi^{ij}_*[S^5_{ij}] \in H_5(V_0)$. By Lemma~\ref{lem:V0} we have
  \[H_5(V_0) \cong H_5(V_0/s(\partial P))  \cong H_5\left(\bigvee_{\ell=1}^5  S_{\ell}^5 \right) \cong \bigoplus_{\ell=1}^5 H_5(S_{\ell}^5)  \cong \ZZ^5.\]
 A choice of an orientation for $\partial P$ and $\S^2$ induces an orientation
for each of the $S_{\ell}^5$ and by the above this fixes a basis for $H_5(V_0)$.
More precisely, we first obtain an orientation on $\partial P \times \S^2$.
Now, recall that $S^5_\ell$ was the one-point compactification of $O_\ell^\circ \times (\S^2\setminus \{*\})$. Hence, the latter can be regarded as an open subset of both $S_{\ell}^5$ and $\partial P \times \S^2$. We obtain an induced orientation via Lemma~\ref{lem:induced-orientation}.

Now, we consider the degree of the composition
  \begin{equation}
S^5_{ij} \to V_0 \to V_0/s(\partial P) \to S^5_\ell.\label{eq:map-of-spheres}
\end{equation}
Recall that $S^5_{ij}$ was obtained as $((H^+_{ij} \cap \partial P) \times C_{ij})/{\hrel}$ 
and  that $C_{ij}$ was contracted to a point $y \in Y_u \cong \PP^1$ by $Y \to Y_u$ for $u \in O_\ell$ and $\ell \neq i,j$. It follows that the map in (\ref{eq:map-of-spheres}) is not surjective in this case and, hence,  has degree $0$. For $\ell=i$ or $\ell=j$ Lemma~\ref{lem:induced-orientation} shows  that the map has degree $1$. Indeed, identifying $C_{ij}$ with $\S^2$ we see that $O_\ell^\circ \times (\S^2\setminus \{*\})$  is mapped homeomorphically onto its image in $S^5_\ell$.

By Lemma~\ref{lem:generators} the classes $\Psi^{ij}_*[D^6_{ij}, S^5_{ij}]$ generate $H_6(V_5,V_0)$.  Hence, the image of $\partial_6$ in $H_5(V_0)$ is generated by the sums $[S_i^5] + [S_j^5] \in \bigoplus_{\ell=1}^5 H_5(S_{\ell}^5)$ with $1 \leq i, j \leq 5$. But this is the kernel of the surjective map
  \[H_5(V_0) \to \ZZ/2\ZZ;\quad [S_i^5] \mapsto 1 \quad \forall_{i \in \{1,\ldots,5\}}.\]
\end{proof}

\bibliography{topquot}
\bibliographystyle{halpha}
\end{document}